\documentclass[12pt]{amsart}
\usepackage{amsfonts,amssymb,amscd,amsmath,enumerate, verbatim,calc,color}

\author{Azadeh Nikou and Anthony G. O'Farrell}

\address[A. Nikou]{Department of Mathematics, Tarbiat Moallem University, 599 Taleghani Avenue, 15618 Tehran, Iran}
\address[A. O'Farrell]{Department of Mathematics and Statistics, NUI, May\-n\-ooth, Co. Kildare, Ireland}
\email[A. Nikou]{a\_nikou@tmu.ac.ir}
\email[A. O'Farrell]{admin@maths.nuim.ie}



\title[Banach algebras of functions]
{Banach algebras of Vector-valued functions}

\newtheorem{thm}{Theorem}[section]
\newtheorem{cor}[thm]{Corollary}
\newtheorem{lem}[thm]{Lemma}

\newtheorem{defn}{Definition}[section]

\newcommand{\C}{\mathbb{C}}
\newcommand{\N}{\mathbb{N}}

\newcommand{\Lip}{\textup{Lip}}
\newcommand{\lip}{\textup{lip}}

\newcommand{\bdy}{\textup{bdy}}

\def\sideremark#1{\ifvmode\leavevmode\fi\vadjust{
\vbox to0pt{\hbox to 0pt{\hskip\hsize\hskip1em
\vbox{\hsize1cm\tiny\raggedright\pretolerance10000
\noindent #1\hfill}\hss}\vbox to8pt{\vfil}\vss}}}

\def\Label#1{\label{#1}{\bf (#1)}~}
\def\Label{\label}

\def\ignore#1{}

\begin{document}

\bibliographystyle{amsplain}

\begin{abstract}We introduce the concept
of an $E$-valued function algebra,
a type of  Banach algebra that  consist
of continuous $E$-valued functions on some compact Hausdorff
space, where $E$ is a Banach algebra.
We present some basic results about
such algebras, having to do with
the Shilov boundary and the set of peak
points of some  commutative $E$-valued
function algebras.
We give some specific examples.
\end{abstract}
\maketitle

\baselineskip20pt

\section{Introduction and Preliminaries}

We consider only algebras over the field
of complex numbers, $\mathbb{C}$.
A Banach algebra is an algebra
equipped with a submultiplicative
norm with respect to which it
is complete.  See \cite{dales, zelazko}
for background on Banach algebras.

\subsection{$E$-valued Function Algebras}
Let $X$ be a nonempty compact Hausdorff space, $E$ be a unital
Banach algebra
and $C(X,E)$ be the
space of all continuous maps from $X$ into $E$. We define the
{\em uniform norm }on $C(X,E)$ by
$$\Vert f\Vert_{X}:={\sup}_{x\in X}\Vert f(x)\Vert,\quad\forall f\in C(X,E). $$
For $f, g \in C(X,E)$ and $\lambda\in \mathbb C$, the pointwise
operations $\lambda f $, $f+g$ and $fg$ in $C(X,E)$ are defined as
usual. It is easy to see that $C(X,E)$, equipped with the norm
$\Vert \cdot \Vert_{X}$, is a
Banach algebra. If $E=\mathbb C$ we get the ordinary uniform function algebra
$C(X)
:=C(X,\mathbb C)$
of all continuous complex-valued functions on
$X$. See any of \cite{browder, dales, gamelin, stout} for background on
uniform algebras.

\begin{defn}
By an {\em $E$-valued function algebra on $X$} we mean
a subalgebra $A\subseteq C(X,E)$, equipped with
some norm  that makes it complete, such that (1)
$A$ has as an
element the constant function $x\mapsto 1_E$, (2) $A$
separates points on $X$, i.e.
given distinct points $a,b\in X$, there exists
$f\in A$ such that $f(a)\not=f(b)$, and
(3) the evaluation map
$$
e_x: \left\{
\begin{array}{rcl}
A&\to& E,\\
f&\mapsto&f(x)
\end{array}
\right.
$$
is continuous, for each $x\in X$.
\end{defn}

We remark that, as it stands, condition (3)
is the very weak assumption
that the inclusion map $A\hookrightarrow E^X$
is continuous, where $E^X$ is given the
cartesian product topology, but
it follows from the Closed Graph Theorem
that if $A$ is an $E$-valued function algebra
on $X$, then the inclusion map $A\hookrightarrow
C(X,E)$ is continuous, so there exists some
constant $M>0$ such that
$$ \|f\|_X \le M\|f\|_A,\ \forall f\in A.$$

Normally, we shall use the same
notation $a$ for the element $a\in E$ and the
constant function $x\mapsto a$ on $X$.


The classical concept of a function algebra (cf. \cite{gamelin,
dales}) corresponds, in our terminology, to  a ${\mathbb C}$-valued
function algebra. Note however that some authors (e.g.
\cite{browder}) have used the term function algebra to refer only to
{\em closed} subalgebras of $C(X)$. We do not assume that an
$E$-valued function algebra on $X$ is closed in the uniform norm.

An important class of examples is afforded
by taking a compact set $X\subset
{\mathbb C}^n$
and a
commutative unital Banach algebra $E$,
and defining
the algebra $P(X,E)$
to be the uniform
closure of $E[z]|X$ in $C(X,E)$, where
$E[z]=E[z_1,\ldots,z_n]$ is the algebra
of all polynomials in the coordinate functions $z_1,...,z_n$ with
coefficients in $E$.
We can also form
the algebra $R(X,E)$, defined to be the uniform closure on
$X$ of the algebra of functions of the form $p(z)/q(z)$,
where $p(z)\in E[z]$, $q(z)\in E[z]$, and
$q(x)\in E^{-1}$ whenever $x\in X$.

Johnson \cite{Johnson} considered the rather similar
concept of the convolution algebra
$L^1(G,A)$ of $A$-valued Bochner-integrable
functions from a locally-compact abelian group
$G$ into a commutative Banach algebra $A$.
The abstract Fourier transform maps such
an $L^1(G,A)$ isomorphically to an $A$-valued 
algebra of continuous functions on the dual group $\hat G$.

\medskip
There is also work \cite{BMS} on operator-valued
Fourier-Stieltjes algebras, and operator-valued
maps occur in applications such as homotopy theory,
but in this paper we are going to concentrate on 
algebras of functions into commutative algebras $E$.  
Specifically, we shall study boundaries.
We proceed to define the terms.


\subsection{Characters}

For a commutative unital Banach algebra $A$, let $M(A)$ denote
the set of all characters (nonzero complex-valued multiplicative
linear functionals) on $A$.  It is well-known that
$M(A)$ is nonempty and that its elements are automatically
continuous, with norm $1$.  Endowed with the
weak-star topology, $M(A)$ becomes a compact
Hausdorff space.  The Gelfand transform of $f\in A$ is the
complex-valued function $\hat{f}$ defined by
$\hat{f}(\varphi)=\varphi(f)$ on $M(A)$.  Let
$\hat{A}=\{\hat{f}: f\in A\}$. The algebra $\hat A$ consists
of $\C$-valued continuous functions on $M(A)$. Hence it is a
$\C$-valued function algebra on $M(A)$ when endowed with
the quotient norm.  However, we shall use the notation
$\|\hat f\|$ to denote the uniform norm of $\hat f$
on $M(A)$, and with respect to this norm $\hat A$ may or
may not be complete.

The kernel of the map
$$ \hat{\ }: \left\{
\begin{array}{rcl}
A &\to& \hat A\\
f&\mapsto&\hat f
\end{array}
\right.
$$
is the Jacobson radical of $A$.  The
characters on $C(X)$ are exactly the evaluations
$e_x:f\mapsto f(x)$, with $x\in X$, and
$X$ is homeomorphic to $M(C(X))$ with its relative weak-star
topology as a subset of the dual $A^*$.

\smallskip
When $A$ is a $\C$-valued function algebra on $X$,
the map $x\mapsto e_x|A$ imbeds $X$ homeomorphically
as a compact subset of $M(A)$.  When this map is surjective,
one calls $A$ a {\em natural} $\C$-valued function
algebra on $X$ \cite{dales}.

The basic example $C(X,E)$ itself was studied
by Hausner, who showed 
that
its maximal ideal space is
homeomorphic to $M(E)\times X$. More precisely,
he showed \cite[Lemma 2]{hausner}:
\begin{lem}[Hausner]\label{L:hausner}
For each commutative Banach algebra $E$
with identity and each compact Hausdorff space
$X$, the map $(\phi,x)\mapsto \phi\circ e_x$ is a
homeomorphism from $M(E)\times X$ onto
$M(C(X,E)$.
\end{lem}

\subsection{Shilov Boundary and Peak Points}
\begin{defn}
A {\em closed boundary }for a commutative
Banach algebra $A$ is a closed subset
$F\subseteq M(A)$ such that for each $a\in A,$
$$\sup_{\varphi\in M(A)}\vert\hat{a}(\varphi)\vert=\sup_{\varphi\in
F}\vert\hat{a}(\varphi)\vert.$$
The {\em Shilov boundary} of $A$ is the intersection
$$  \Gamma(A) = \bigcap\{ F: F\textup{ is a closed boundary for }A\}.$$
\end{defn}
It can be shown \cite[Theorem 15.2]{zelazko} or \cite{stout}
that $\Gamma(A)$ is the unique minimal
closed boundary for $A$.

\begin{defn}
Let $A$ be a unital commutative Banach algebra. A closed subset
$S\subseteq M(A)$ is called a {\em peak set }if there exists an element
$a\in A$ such that $\hat{a}(\varphi)=1$ for $\varphi\in S$ and
$\vert\hat{a}(\psi)\vert<1$ for $\psi\in M(A)\setminus S$. A point
$\varphi\in M(A)$ is a {\em peak point }for $A$ if $\{\varphi\}$ is a peak
set. We write $S_0(A)$ for the set of peak points for $A.$
\end{defn}
Obviously, $S_0(A)\subseteq \Gamma(A)$.  If
$M(A)$ is metrisable, then (cf. \cite[Cor. 4.3.7]{dales})
$\Gamma(A)$ is the closure of $S_0(A)$.

\subsection{Main Result}

Our results are about commutative algebras.

In Section \ref{S:2}  we introduce the concept
of an {\em admissible quadruple} $(X,E,B,\tilde B)$,
which formalises the idea of an $E$-valued
function algebra $\tilde B$ that is organically connected
to a $\C$-valued function algebra $B$ on the same space $X$.
To such a quadruple we associate an injective map
$\pi:M(E)\times X\to M(\tilde B)$,
and we say that the quadruple is
{\em natural} when $\pi$ is bijective.
We prove the following result about the
relation between the three Shilov boundaries
that are in play:

\begin{thm}\label{T:main}
Let $(X,E,B,\tilde{B})$ be a natural admissible
quadruple.
Then the associated map
$\pi$ maps
$\Gamma(E)\times \Gamma(B)$ homeomorphically onto
$\Gamma(\tilde{B})$
\end{thm}

We give some specific examples, and other results.

\section{Admissible Quadruples}\Label{S:2}

\begin{defn}
By an {\em admissible quadruple} we mean a quadruple
$(X,E,B,\tilde{B})$, where
\begin{enumerate}
\item
$X$ is a compact Hausdorff space,
\item $E$ is a commutative Banach algebra with unit,
\item $B\subseteq C(X)$ is a natural $\C$-valued function
algebra on $X$,
\item $\tilde{B}\subseteq C(X,E)$ is an $E$-valued function
algebra on $X$,
\item $B\cdot E \subseteq \tilde{B}$, and
\item\label{last}$\{\lambda\circ f, f\in \tilde{B}, \lambda\in M(E)\}\subseteq B.$
\end{enumerate}
\end{defn}

We remark that if we assume that the linear span
of $B\cdot E$ is dense in $\tilde B$,
then (\ref{last}) is automatically true.

Condition (\ref{last}) is undemanding if the Jacobson
radical $J(E)$ of $E$ is large.  In fact, we are mainly interested
in semisimple algebras. The meat of Theorem \ref{T:main}
is really about the quotient $E/J(E)$.

Given an admissible quadruple $(X,E,B,\tilde{B})$,
we define the {\em associated map}
$$\pi:\left\{
\begin{array}{rcl}
M(E)\times X &\to& M(\tilde{B})\\
(\psi,x) &\mapsto& \psi\circ e_x.
\end{array}
\right.
$$
\begin{lem}
Let $(X,E,B,\tilde{B})$ be an admissible quadruple.
Then the associated map $\pi$ is a continuous injection.
\end{lem}
\begin{proof}
$\pi$ is injective from $M(E)\times X$
into $M(\tilde B)$, since $\hat E$ separates points on
$M(E)$ and $\hat B$ separates points on $X$.
To see that $\pi$ is continuous, observe
that it is the composition of the (weak-star continuous)
restriction map $C(X,E)^*\to
(\tilde B)^*$ with Hausner's homeomorphism
$M(E)\times X\to M(C(X,E))$.
\end{proof}

\begin{cor}
Let $(X,E,B,\tilde{B})$ be an admissible
quadruple. Then the following are equivalent:
\\
(1)
The associated map $\pi$ is surjective.\\
(2)
The associated map $\pi$ is bijective.\\
(3)
The associated map
$\pi$ is  a
homeomorphism of
$M(E)\times X$ onto
$M(\tilde{B})$
\end{cor}

\begin{defn}
We say that an admissible quadruple $(X,E,B,\tilde{B})$
is {\em natural} if the associated map $\pi$
is bijective.
\end{defn}

For instance, if $B$ is a natural $\C$-valued function algebra on $X$,
then $(X,\C,B,B)$ is a natural admissible quadruple,
so this terminology is a reasonable extension
of the usual use of \lq natural'.
Further, if $(X,E,B,\tilde B)$
is an admissible quadruple and $E$ is semisimple, then $\hat E$ (with
the induced norm given by $\|\hat h\|=\|h\|_E$) is
a natural $\C$-valued function algebra on $M(E)$,  so $\tilde B$
is isometrically isomorphic to a $\C$-valued function algebra on
$M(E)\times X$, and it is a natural $\C$-valued
function algebra if and only if the quadruple is natural.

Tomiyama \cite{Tom} showed that if
$A$ and $B$ are commutative
Banach algebras with identity, and some completion
of $C$ of $A\otimes B$ is also a Banach algebra,
then the natural map $M(A)\times M(B) \to M(C)$ 
is a homeomorphism.  Thus if $(X,E,B,\tilde B)$
is an admissible quadruple, and the linear span of $B\cdot E$
is dense in $\tilde B$, we may apply
Tomiyama's theorem with $A=E$ and $B=B$ and deduce that
the quadruple is natural.

In view of the corollary, when given
a natural admissible quadruple
$(X,E,B,\tilde{B})$, we often
identify $M(E)\times X$ with $M(\tilde B)$.

\medskip
\begin{proof}[Proof of Theorem \ref{T:main}]
First, we show that the image of $\pi$ is a boundary for $\tilde B$:
Let $f\in \tilde{B}$.  Fix a character $\phi\in M({\tilde{B}})$.
Then $\phi$ is of
the form $\psi\circ e_x$ for some $x\in X$ and some
$\psi\in M(E)$, and then $\hat f(\phi)=(\psi\circ f)(x)$.
Now $\psi\circ f\in B$, so
there exists
a point $y\in \Gamma(B)$ such that
$
\vert(\psi\circ f)(x)\vert
\le \vert(\psi\circ f)(y)\vert$.
Next, $(\psi\circ f)(y)= \widehat{f(y)}(\psi)$, and $f(y)\in E$, so
there exists a point $\chi\in \Gamma(E)$
such that $|\widehat{f(y)}(\psi)|\le |\widehat{f(y)}(\chi)|$.
Thus
$$ |\hat f(\phi)| \le |\hat f(\pi(\chi,y))|.$$

This shows that for each $f\in\tilde B$,
$\hat f$ attains its maximum
modulus on the image of $\pi$, so that image is a boundary, and
$$\Gamma(\tilde B)
\subseteq \pi(\Gamma(E)\times \Gamma(B)).$$

\noindent
To see the opposite inclusion,
fix $x\in \Gamma(B)$ and $\psi\in \Gamma(E)$. Let
$U$
 be any neighborhood of $x$  in $X$ and $V$ be any
 neighborhood of $\psi$ in $M(E).$
There exists $f\in B$ such that $\|\hat{f}\|=1$ and $\vert
f(y)\vert<1$ for all $y\in X\backslash U$. In addition there
exists $v\in E$ such that $\|\hat{v}\|=1$ and $\vert\phi(v)\vert<1$
for all $\phi\in M(E)\setminus V$. Now define $g:X\to E$
by $g=vf$. We have $g\in \tilde{B}$ and
$$\begin{array}{rcl}
\|\hat{g}\| &=& \sup_{\phi\in M(E)}\sup_{y\in X}
\vert\widehat{vf}(\phi\circ e_y)\vert\\
&=&\sup_{\phi\in M(E)}\sup_{y\in X}\vert f(y)\phi(v)\vert\\
&=&\sup_{\phi\in M(E)}
\vert\phi(v)\vert
\cdot\sup_{y\in X}
\vert f(y)\vert
\\
&=&
\|\hat{v}\|.\|\hat{f}\|= 1.
\end{array}$$
On the other hand
every $\varphi'\in\pi(M(E)\times X\setminus (U\times
V))$ is of the form $\varphi^{'}=\phi\circ e_{y}$
with $y\in X\setminus U$ or $\phi\in
M(E)\setminus V$ (or both). Therefore,
$$\vert\hat{g}(\varphi^{'})\vert=
\vert\varphi^{'}(vf)\vert
=
\vert\phi(v)f(y)\vert<1.$$
Since $U$ and $V$ were
arbitrary neighbourhoods, it follows from \cite[Theorem 15.3]{zelazko}
that
$\psi\circ e_x\in\Gamma(\tilde{B})$. Therefore $\pi(\Gamma(B)\times
\Gamma(E))\subseteq\Gamma(\tilde{B})$ and so the proof is complete.
\end{proof}

\subsection{Examples}

$\mathbf{(i)}$ Let $X$ be a compact Hausdorff space and $E$ be a
unital commutative Banach algebra. Then $(X,E, C(X), C(X,E))$ is an
admissible quadruple.  It is natural by Lemma \ref{L:hausner},
and this case of Theorem \ref{T:main} is Hausner's theorem
\cite{hausner} that the Shilov boundary of $C(X,E)$ is equal to the
cartesian product $X\times \Gamma(E)$.

$\mathbf{(ii)}$ Let  $(X,d)$ be a compact metric space
and $E$ be a commutative unital Banach algebra.
For a constant $0<\alpha\leq 1$ and a
function $f:X\rightarrow E$, the Lipschitz constant of $f$ is
defined as
$$p_{\alpha}(f):=\sup_{
\genfrac{}{}{0pt}{1}%
{x,y \in X}%
{x\neq y} } \frac{\| f(x)-f(y)\|}{d(x,y)^{\alpha}},$$ and the
$E$-valued big Lipschitz algebra or simply $E$-valued Lipschitz
algebra (of order $\alpha$) is defined by
$$\Lip^{\alpha}(X,E)=\left\{f:X\rightarrow E :\\
 p_{\alpha}(f)<\infty\right\}.$$
Similarly, for $0<\alpha<1$, the $E$-valued little Lipschitz
algebra (of order $\alpha$) is defined by
$$\lip^{\alpha}(X,E)=\\
\left\{f\in \Lip^{\alpha}(X,E) :\frac{\| f(x)-f(y)\|}{d(x,y)^{\alpha}}\rightarrow
0\textup{ as }d(x,y)\rightarrow 0 \right\}.$$
For each $f\in
\Lip^{\alpha}(X,E)$ we define a norm by
$$\Vert f\Vert_{\alpha} = \Vert f\Vert_{X} + p_{\alpha}(f).$$

In \cite{CZX} it was shown that $(\Lip^{\alpha}(X,E),\Vert \cdot
\Vert_{\alpha})$ is a Banach algebra
having $\lip^{\alpha}(X,E)$ as a closed subalgebra.
It is relatively straightforward to check that $(X,E,
\Lip^\alpha(X,\C), \Lip^\alpha(X,E))$ is an admissible quadruple,
for each $\alpha\in(0,1]$, and that $(X,E, \lip^\alpha(X,\C),
\lip^\alpha(X,E))$ is an admissible quadruple, for each
$\alpha\in(0,1)$. (The result that the maximal ideal space of
$\Lip^{\alpha}(X)$ is $X$ is due originally to Sherbert
\cite{sherbert, dales}.)

The scalar-valued Lipschitz algebras are normal, because if $F$ and
$K$ are disjoint nonempty closed subsets of $X$, then the function
$\displaystyle f:x\mapsto \frac{d(x,F)}{d(x,F)+d(x,K)}$ belongs to
$\Lip^1(X)$. It follows \cite[p. 413]{dales} that they have
partitions of unity subordinate to any open covering. Applying
partitions of unity and a method similar to Hausner's in \cite[Lemma
1]{hausner}, one can see that  each of these $E$-valued  Lipschitz
algebras (and little Lipschitz algebras) is dense in $C(X,E)$. Then,
given a character  $\phi$  on $\Lip^{\alpha}(X,E)$ and a function
$f\in C(X,E)$, we may choose a sequence $(f_n)\in \Lip^\alpha(X,E)$
such that $\|f-f_n\|_X\to0$.  Since
$$ \|\hat h\|_{M(\Lip^\alpha(X,E))}\le \|h\|_X$$
for each $h\in \Lip^\alpha(X,E)$, the sequence
$(\phi(f_n))$ is Cauchy, so we may define
$\tilde{\phi}(f)=\lim_n\phi(f_n)$.  Clearly
$\tilde\phi(f)$ does not depend on the choice
of $(f_n)$, and $\tilde\phi$ is a well-defined
character on $C(X,E)$, extending $\phi$.  Thus,
by Lemma \ref{L:hausner},
$\phi=\psi\circ e_x$ for some $\psi\in M(E)$ and some $x\in X$.
A similar argument works for $\lip^\alpha(X,E)$.
Thus Theorem \ref{T:main} applies, and
the Shilov boundary of $\Lip^{\alpha}(X,E)$(or
$\lip^{\alpha}(X,E))$ is equal to the cartesian product $X\times
\Gamma(E)$ in the product topology.

$\mathbf{(iii)}$ Let $X$ be a compact set in $\mathbb C^n$ and $E$
be a unital commutative Banach algebra, and consider the algebra
$P(X,E)$.   The algebra $P(X)=P(X,\C)$ has character space naturally
identified with $\hat X$, the polynomially-convex hull of $X$
\cite{browder, gamelin, leibowitz, stout}, and $P(X,E)$ may be
regarded as an $E$-valued function algebra on $\hat X$. Using this,
it is easy to see that $(\hat X, E, P(X), P(X,E))$ is an admissible
quadruple: In fact, each $f\in P(X,E)$ is the limit in norm of a
sequence $\{g_n\}$ with each $g_n$ of the form
$\sum_{j=1}^{m_n}a_jp_j$ where $m_n\in\N$, $a_j\in E$ and $p_j\in
P(X)$ depend on $n$. Then by a method similar to \cite[Proposition
1.5.6]{kaniuth} one sees that $P(X,E)=P(X)\check{\otimes} E$. Thus,
since $P(X)=P(\hat X)$ \cite[Chapter II, Theorem 1.4]{gamelin}, we
have
$$ P(X,E)=P(X)\check{\otimes} E=P(\hat X)\check{\otimes} E=P(\hat X,E).$$
We note that in the particular case when $X$ is a
compact plane set, $\hat X$
is obtained by \lq\lq filling in the holes" in
$X$, and $\Gamma(P(X))$ is the topological boundary
of $\hat X$ in $\C$.  In higher dimensions, the Shilov
boundary $\Gamma(P(X))$ is some closed subset of
$\bdy(X)$.

By a method
similar to \cite[Lemma 2]{hausner}, one sees that every
character $\phi$ on $P(X,E)$ is of the form $\phi=\psi\circ e_x$
for some $x\in \widehat{X}$ and some $\psi\in M(E)$. Therefore the
theorem applies, and
the Shilov boundary of $P(X,E)$ is equal to the cartesian product
$\Gamma(E)\times\Gamma(P(X))$

$\mathbf{(iv)}$ Let $X\subset\mathbb{C}$ be compact, $E$ be a
commutative unital Banach algebra and $E^*$ be the dual space of
$E.$ The algebra of $E$-valued analytic functions is defined as
follows,
$$A(X,E)=\{f\in C(X,E) : \Lambda \circ f \in A(X), \Lambda\in E^*\},$$
where $A(X)$ is the algebra of all continuous functions on $X$ into
$\mathbb C$ which are holomorphic on the interior of $X$. It is clear
that $A(X,E)$ is a closed subalgebra of $(C(X,E),\|\cdot\|_X)$.
Arens showed \cite{gamelin} that $M(A(X))$
is naturally identified with $X$, and so one sees at once
that $(X,E, A(X), A(X,E))$ is an admissible quadruple.
By a
method similar to the one given in \cite[Theorem 2]{character}, we
can deduce that when $E$ is a
unital Banach algebra then every character $\phi$ on $A(X,E)$ is of
the form $\phi=\psi\circ e_x$ for some $x\in X$ and some
$\psi\in M(E)$. So the Shilov boundary of $A(X,E)$ is equal to the
cartesian product $\Gamma(A(X))\times \Gamma(E)$ in the product
topology, by Theorem \ref{T:main}.

$\mathbf{(v)}$ Theorem \ref{T:main} also applies to the algebra $R(X,E)$,
for any commutative unital Banach algebra $E$.
The characters on $R(X)$ are the evaluations at the points
of the rationally-convex hull $\check X$ of $X$, which is the set
of points $a\in\C^n$ such that each polynomial $p(z)\in\C[z]$
that vanishes at $a$ also vanishes at some point of $X$.
In dimension $n=1$, $\check X=X$, but in higher dimensions
it may be a larger set.  So $R(X)$ is a natural
$\C$-valued function algebra on $\check X$.
 
We claim that every character $\phi$ on $R(X,E)$
is of the form $\phi=\psi\circ e_x$, for some
$\psi\in M(E)$ and some $x\in\check X$.

To see this, let $\phi\in M(R(X,E))$. The restriction of $\phi$ to
$P(X,E)$ is a character, so there exists $x_0\in \hat{X}$ and
$\psi\in M(E)$ such that $\phi=\psi\circ e_{x_0}$ on $P(X,E)$.
Given $g=p/q$ where $p,q\in E[z]$ and $q(x)\in E^{-1}$ for each
$x\in X$, we get $p=gq$, $\phi(p)=\phi(g)\phi(q)$, and hence (since
$\phi(q) = \psi(q(x_0))\not=0$)
$$ \phi(g) = \frac{\psi(p(x_0))}{\psi(q(x_0))} = \psi(g(x_0)).$$
Thus, by continuity, $\phi=\psi\circ e_{x_0}$ on
all $R(X,E)$.  Since $R(X)1\subseteq R(X,E)$, it follows that
$x_0\in\check X$ (cf. \cite[Theorem 5, p. 86]{gamelin}.
Thus the claim holds.

Hence if $X$ is rationally-convex, then
$(X, E, R(X), R(X,E))$ is
a natural admissible quadruple.

There seems no reason to suppose that $R(X,E)=R(\check X,E)$
for general $X$, except when $E$
is a uniform algebra.  In general, one sees readily that
there is a contractive algebra homomorphism
$$ R(X,E) \to C( \check X, C(M(E))),$$
and that if $E$ is a uniform algebra then
this gives an isometric isomorphism from
$R(X,E)$ onto $R(\check X,E)$. We do not, however, know 
an example in which the restriction map 
$R(\check X,E) \to R(X,E)$ is not onto.

$\mathbf{(vi)}$ The bidisk algebra \cite{browder} may 
(in view of Hartog's theorem) be regarded
as $\tilde B=A(X,A(X))$, where $X$ is the closed unit disk in $\C$.
The quadruple $(X,A(X),$$A(X),$$\tilde B)$ is admissible, the theorem
applies, and reduces to the classical fact that the Shilov
boundary of $\tilde B$ is the torus. More generally, one gets the
(known) result that the Shilov boundary of $A(X,A(Y))$
is $\bdy X\times \bdy Y$ whenever $X\subset \C$
and $Y\subset\C$ are compact.

$\mathbf{(vii)}$ Let $0<\alpha<1$. The subalgebra of 
$\lip^{\alpha}(X,E)$ which is the closure of $E[z]|X$ in
$\Lip^{\alpha}(X,E)$ norm, where $X\subset
\mathbb C^n$, is denoted by $\Lip^{\alpha}_P(X,E)$. 
It is easy to see that
$\Lip^{\alpha}_P(X,E)$ is dense in
$P(X,E).$ Now by \cite[p.15]{honary}, $M(\Lip^{\alpha}_P(X,\C))$
$=\hat{X}$. 
Thus if $X$ is
polynomially-convex, then the quadruple
$(X,E,\Lip^{\alpha}_P(X,\C),
\Lip^{\alpha}_P(X,E))$
is admissible and natural.

$\mathbf{(viii)}$ Also for $0<\alpha<1$, the subalgebra of
$\lip^{\alpha}(X,E)$ which is the closure of the algebra of
functions of the form $p(z)/q(z)$ in $\Lip^{\alpha}(X,E)$
 where $X\subset \mathbb C^n$, $p(z)\in E[z]$,
$q(z)\in E[z]$, and $q(x)\in E^{-1}$ whenever $x\in X$, is denoted
by $\Lip^{\alpha}_R(X,E)$. It is
easy to see that $\Lip^{\alpha}_R(X,E)$
 is dense in $R(X,E)$ Now by
\cite[p.15]{honary}, $M(\Lip^{\alpha}_R(X,\C))$
$=\check{X}$.  Thus if $X$ is rationally-convex, then
the quadruple
$(X,E,\Lip^{\alpha}_R(X,\C),
\Lip^{\alpha}_R(X,E))$
is admissible and natural.
\subsection{}
By similar arguments, one obtains the following:
\begin{thm}\label{peak Bx}
Let $(X,E,B,\tilde{B})$ be a natural admissible
quadruple.
 Then
the set of peak points of $\tilde{B}$ is equal to the cartesian product
$S_0(B)\times S_0(E)$ in the product topology, that is,
$$S_0(\tilde{B})=S_0(B)\times S_0(E).$$
\end{thm}
\subsection{Acknowledgments}
This paper was conceived while the first-named author was visiting
National University of Ireland, Maynooth, whose hospitality is
acknowledged with thanks.

The authors are grateful to the referee, whose observations and suggestions
greatly improved the paper.

\end{document}